\documentclass{article}

\usepackage{amsfonts}
\usepackage{amsmath}
\usepackage{mathtools}
\usepackage{amsthm}
\usepackage{amssymb}
\usepackage{booktabs}
\usepackage{wrapfig}
\usepackage{numprint}
\usepackage{todonotes}
\usepackage[round]{natbib}  
\usepackage{dblfloatfix} 
\let\on=\operatorname
\newcommand{\ud}{\,\mathrm{d}}

\newcommand{\R}{\ensuremath{\mathbb{R}}}
\newtheorem{proposition}{Proposition}
\newtheorem{definition}{Definition}
\newtheorem{corollary}{Corollary}
\newtheorem{lemma}{Lemma}

\newtheorem{assumption}{Assumption}

\begin{document}

%

%

\author{Adrien Vacher$^{\star \circ}$ \\ 
        \and François-Xavier Vialard$^{\star \circ}$ }

\date{
{\small 
		    $^\star $ Laboratoire Informatique Gaspard Monge
Univ Gustave Eiffel, CNRS, LIGM, F-77454 Marne-la-Vallée, France \\
			$^\circ$ INRIA Paris, 2 rue Simone Iff, 75012, Paris, France  \\
		\texttt{ adrien.vacher@u-pem.fr, francois-xavier.vialard@u-pem.fr}
}\\
}

\title{Stability and upper bounds for statistical estimation of unbalanced transport potentials.}

\maketitle

\begin{abstract}
In this note, we derive upper-bounds on the  statistical estimation rates of unbalanced optimal transport (UOT) maps for the quadratic cost. Our work relies on the stability of the semi-dual formulation of optimal transport (OT) extended to the unbalanced case. Depending on the considered variant of UOT, our stability result interpolates between the OT (balanced) case where the semi-dual is only locally strongly convex with respect the Sobolev semi-norm $\dot{H}^1$ and the case where it is locally strongly convex with respect to the $H^1$ norm. When the optimal potential belongs to a certain class $C$ with sufficiently low metric-entropy, local strong convexity enables us to recover super-parametric rates, faster than $1/\sqrt{n}$.
\end{abstract}

\section{Introduction}
In its original formulation, OT is a tool to compare probability distributions: it seeks a map that optimally transports one distribution $\mu$ to an other distribution $\nu$ with respect to some fixed cost $c$ and it returns the associated transport cost. This problem was later relaxed into a linear program by Kantorovitch and its primal formulation consists into seeking a coupling instead of a map with minimal cost and whose marginals are constrained to be $\mu$ and $\nu$; under suitable assumptions on the measures and the cost, this relaxation is tight \citep{Brenier1991}. Quite recently, OT was extended to arbitrary positive measures \citep{chizat2017unbalanced}, with possibly different masses, thus the name Unbalanced Optimal Transport (UOT). On the primal problem, the hard marginal constraints are relaxed by soft entropic penalties. From an applied point of view, the mass constraint relaxation is indeed a desirable feature: it allows not only displacement of mass but also local growth or shrinkage \citep{sejourne2019sinkhorn}. In image processing for instance, it can remove or at least decreases blurred areas in favor of sharper contrasts \cite{feydy2017}. From a statistical point of view, UOT may appear as a more robust version of OT as it is able to cut down the outliers. 
At the heart of classical OT, rather than the mere OT distance value, the main object of interest is the transport map: in generative imaging, we use the transport map to generate images from noise \citep{liu2019wasserstein}, for point cloud approximation, the particle flow is driven by the maps \citep{merigot2021non} and in Domain Adaptation, the source distribution is transported on the target using an OT map \citep{courty2017domain}.
Novel applications include predicting the evolutions of cells from measurements  \cite{schiebinger2019optimal,Yang2020PredictingCL}.
Notably in the case of a quadratic cost, Brenier showed that these maps are linked with the dual formulation of Kantorovitch relaxation: the map from $\mu$ to $\nu$ corresponds to the the gradients of the first variable of the dual problem that we shall refer to as a \textit{potential}. This potential is the solution of a linear program, yet with an infinite dimensional constraint. Hence, even in the case where the measures $\mu$ and $\nu$ are known analytically, there is in general no closed form to compute the OT potentials. In particular, recent methods instead rely on $n$-samples empirical counterparts of $\mu$ and $\nu$ to statistically estimate the cost and the potentials \citep{genevay2016stochastic, seguy2018large}. 
Such methods are thus concerned with the statistical estimation of optimal transport quantities, such as the cost or the potentials, see \cite{2019Panaretos,bookPanaretos} for an overview of this rapidly growing field.
If the distributions $\mu$, $\nu$ are only assumed to have a density w.r.t. the Lebesgue measure, the error achieved by the $\textit{plug-in}$ estimator, which is simply the OT between the empirical samples, scales in $O(n^{-\frac{2}{d}})$ \citep{chizat2020faster}. In particular, without further assumptions the OT problem is said to suffer the curse of dimension. However, in the seminal work of \citet{hutter2021Minimax}, the authors showed that if the original potential is $\alpha$-smooth, then its gradient could be at best estimated in $n^{-\frac{2(\alpha+1)}{2\alpha + d}}$ with respect to the squared $L^2$ distance. They also provided an estimator that actually achieved this rate of estimation, hence providing a \textit{minimax} rate under smoothness assumptions. Yet, we emphasize the fact that their estimator as such in infeasible as it requires in particular to project on the space of convex, $k$-times differentiable functions. This result has triggered follow-ups on computationally feasible and efficient estimators of the optimal transport maps, for instance \cite{muzellec2021nearoptimal} leveraging the underlying smoothness of the optimal maps and \cite{pooladian2021entropic,2022arXiv220208919P,2021arXiv210701718D} using entropic regularization.

In this note, we propose to explore the results of statistical estimation of transport potentials to the unbalanced setting and derive upper-bounding rates. In particular, we restrict ourselves to probability measures instead of positive measures, yet it does not affect the relevance of using UOT instead of classical OT as it allows to handle outliers. Instead of making explicit smoothness assumptions on the ground truth, we assume that it belongs to a certain class $C$ and derive rates of estimation depending on the complexity of $C$, namely its metric entropy. In particular, our statistical analysis relies on an unbiased estimation, where our search space for the empirical candidate is the same set $C$, that is assumed to contain the ground truth. Using the recent regularity results on Unbalanced Optimal transport \citep{gallouet2021regularity}, we shall in particular cover the case where the smoothness assumption is not directly made on the potential $z_0$ but instead on the measures $\mu$, $\nu$. As in the balanced case, we rely on the semi-dual formulation of UOT for which we derive stability results; interestingly, thanks to the extra convexity brought by the entropic relaxation of the marginal in some cases, we do not need to assume smoothness of the potentials to derive those stability estimates. In the case where the metric entropy of $C$ slowly diverges, the strong convexity enables us to use localization arguments and derive super-parametric rates. In particular, we obtain two different regimes that depend on the metric entropy of $C$; under smoothness assumptions, our rates closely match those of \citet{hutter2021Minimax} in the highly smooth case. 

\paragraph{Assumptions and notations} In this paper $X, Y$ are compact subsets of $\mathbb{R}^d$, $\mu$ and $\nu$ are positive measures over $X$ and $Y$ respectively with their $n$-independent samples empirical counterparts $\hat{\mu}$, $\hat{\nu}$ when $\mu$ and $\nu$ are probability measures. We shall denote by $\text{supp}(\mu)$, $\text{supp}(\nu)$ the support of $\mu$ and $\nu$ respectively. We shall denote by $\langle \cdot, \cdot \rangle$ the pairing between radon measures and continuous functions and by $q$ the quadratic function $q(x) = \frac{1}{2}\|x \|^2$. The notation $\|\cdot \|_{L^p_Z}$ for $p \in [1, + \infty ]$ shall refer to the $L^p$ norm over functions defined on Borel sets $Z$ as $\|f\|_{L^p_Z} = (\int_Z |f(x)|^p  \ud x)^{\frac{1}{p}}$. Conversely, for a probability measure $\beta$, we shall denote for $p \in [1, + \infty ]$, $\|g\|_{L^p(\beta)} = (\int_{x} |g(x)|^p \ud \beta(x))^{\frac{1}{p}}$.

\section{Unbalanced Optimal Transport}
In this section, we present unbalanced optimal transport via primal and dual formulations. The latter is used to show stability estimates, generalizing standard strong convexity estimates for the semi-dual in optimal transport. 
Unbalanced optimal transport (UOT) is a generalization of standard optimal transport which relaxes the marginal constraints using a convex divergence between positive measures. 
The primal formulation of UOT uses Csiz\'ar divergences which are defined as follows.
\begin{definition}[Csiz\'ar divergences]
Let $F: \R_+ \mapsto \R_+ \cup \{+\infty \}$ be a convex lower semicontinuous function such that $F(1) = 0$. Its recession constant is $F^{'}_{\infty} = \lim_\infty \frac{F(r)}{r}$.
Let $\mu,\nu$ be non negative Radon measures on a convex domain $\Omega$ in $\R^d$. The {\it Csiszàr divergence} associated with $F$ is
\begin{equation}
D_F(\mu,\nu) = \int_\Omega F\left(\frac{\ud \mu(x)}{\ud \nu(x) }\right) \ud \nu(x) + F^{'}_{\infty} \int_\Omega\ud \mu^{\perp}\,,
\end{equation}
where $\mu^{\perp}$ is the orthogonal part of the Lebesgue decomposition of $\mu$ with respect to $\nu$.
\end{definition}
The Kullback-Leibler divergence is obtained for $F(x) = x\log(x) - x + 1$.
The primal formulation of UOT is defined by, for
$\rho_0,\rho_1 \in \mathcal{M}_+(\Omega)$, 
\begin{equation*}\label{EqPrimalUOT}
    \on{UOT}(\rho_0,\rho_1) = \inf_{\gamma \in \mathcal{M}_+(X\times Y)} D_{F_{0}}(\gamma_0,\rho_0) + D_{F_1}(\gamma_1,\rho_1) + \int_{X\times Y} c(x,y) \ud \gamma(x,y)\,,
\end{equation*}
where $F_0,F_1$ are two possibly different entropy functions. Note that standard OT is recovered for the entropy function $F(x) = \iota_{\{1\}}(x)$ the convex indicator function of $\{ 1\}$. 

The optimization problem associated with $\on{UOT}$ is convex and its dual formulation reads, denoting $F^*$ the Legendre transform of $F$,
\begin{equation}\label{EqDualUOTProblem}
    \sup_{z_0,z_1 \in C_b(X), C_b(Y)} -\int_X F_0^*(-z_0(x)) \ud \rho_0(x) -\int_Y F_1^*(-z_1(y)) \ud \rho_1(y)
\end{equation}
under the constraint \begin{equation}\label{EqInequalityConstraint}
    z_0(x) + z_1(y) \leq c(x,y)\,.
\end{equation}
The following proposition shows that at optimality, $z_0$ is a standard optimal transport potential between modified versions of $(\mu, \nu)$ for the cost $c(x,y)$.
\begin{proposition}[see Lemma 3 in \citet{gallouet2021regularity}] Assume that $\varphi^*$ is differentiable on its domain. At optimality of \eqref{EqDualUOTProblem}, the pair $(z_0, z_1)$ reads $z_1 = z_0^c \coloneqq \inf_{x} c(x, \cdot ) - z_0(x)$ and $z_0$ is an optimal transport potential between $\tilde{\mu} = \partial F_0^*(-z_0) \mu$ and $\tilde{\nu} = \partial F_1^*(-z_1) \nu$. 
\end{proposition}
In the rest of the paper, we shall from now on focus assume that the cost is quadratic $c(x, y) = q(x-y)$. Furthermore, to avoid heavy notations, we shall assume $F_0 = F_1 \coloneqq \varphi$. However, similar results can be obtained when the entropy functions are different.

It is possible to optimize Formula \eqref{EqDualUOTProblem} with respect to the second variable to obtain the so-called \emph{semi-dual} formulation of UOT. Indeed, the optimal $z_1$ is the $c$-conjugate of $z_0$. Using this argument on $z_0^c$, one can further assume that $z_0$ is the $c$-conjugate of a function $z_1$, which says that  $\tilde{z}_0 = q - z_0$ is convex. In this case, the semi-dual UOT problem can be read as
\begin{equation}\label{EqDualUOTProblemBrenier}
   - \inf_{z \in C_b(X), z \text{ convex}} \langle \varphi^*(z - q), \mu \rangle  + \langle \varphi^*(z^* - q), \nu \rangle \, ,
\end{equation}
where we injected $(q - z)^c = q - z^*$. Thus, let us introduce
\begin{definition}[Semi-Dual UOT]
Given nonnegative measures $\mu,\nu$, the UOT semi-dual is defined by
\begin{equation}
    J_{\mu, \nu}(z) = \langle \varphi^*(z - q), \mu \rangle  + \langle \varphi^*(z^* - q), \nu \rangle \, .
\end{equation}
\end{definition}
When confusion is possible, we shall denote $J_{\mu, \nu}$ by $J$. This semi-dual objective $J$ remains convex and even gains in convexity with respect to the original objective. This phenomenon is well-known in standard OT and we show how it extends in the unbalanced setting. The important difference with standard OT is that when $\varphi^*$ is strongly convex, the stability is expressed in an $H^1$ norm instead of the $L^2$ norm of the gradient.


\begin{proposition}[Stability estimate]
\label{propStability}
The semi-dual functional $J$ is convex. Assume that $\varphi^*$ is differentiable and that $\nu$ is absolutely continuous with repsect to the Lebesgue measure. Let $z$ a $\lambda$-strongly convex potential and the optimal potential $z_0$. Then, it holds
\begin{equation}
J(z) - J(z_0) \geq \frac{1}{2\lambda} \mathbb{E}_{\tilde{\nu}}[\| \nabla (z_0^* - z^*)\|^2] + C_{z^*} \mathbb{E}_{\nu}[(z_0^* - z^*)^2] +  C_z \mathbb{E}_{\mu}[(z_0 - z)^2] \,,
\end{equation}
with nonnegative constants $C_{z}$ and $C_{z^*}$ depends on $\varphi^*$ and $z$. If $z$ and $z^*$ and uniformly bounded on the support of $\mu$ and $\nu$ respectively and if $\varphi^*$ is strongly convex on every compact, then $C_z$ and $C_{z^*}$ are uniformly lower bounded by a constant $C' > 0$.
\end{proposition}
\begin{proof}
We start by applying the convexity inequality
\begin{equation}
    \varphi^*(y)  - \varphi^*(x) \geq {\varphi^{*}}'(x) (y - x) + \frac 12 m_{xy} |x -y|^2\,,
\end{equation}
where $m_{xy} \geq 0$. We apply it in the difference of dual values and we get
\begin{align*}\label{EqDifferences}
    J(z) - J(z_0) \geq & \langle z - z_0, (\varphi^{*})'(z_0 - q) \mu \rangle +  \langle z^* - z_0^*, (\varphi^{*})'(z_0^* - q) \nu \rangle \\
     & + C_z \langle (z_0 - z)^2, \mu \rangle + C_{z^*} \langle (z_0^* - z^*)^2, \nu \rangle \,,
\end{align*}
where 
\begin{equation}
    \begin{cases}
    C_z = \underset{x \in \text{supp}(\mu)}{\inf}  m_{z(x) - q(x), z_0(x) - q(x)}\\
    C_{z^*} = \underset{y \in \text{supp}(\nu)}{\inf}  m_{z^*(y) - q(y), z_0^*(y) - q(y)} \, .
    \end{cases} 
\end{equation}
Now, recall that $z_0$ is an optimal potential for the transport of measure $\tilde \mu \coloneqq {\varphi^{*}}'(z_0-q)\mu$ onto the measure $\tilde \nu \coloneqq {\varphi^{*}}'(z_0^*-q)\nu$. Denoting $\tilde{J}(z) = \langle z, \tilde{\mu} \rangle + \langle z^*, \tilde{\nu} \rangle$, Equation \eqref{EqDifferences} reads 
\begin{equation}\label{EqDifferences}
    J(z) - J(z_0) \geq \tilde{J}(z) - \tilde{J}(z_0) +  C_z \langle (z_0 - z)^2, \mu \rangle + C_{z^*} \langle (z_0^* - z^*)^2, \nu \rangle  \,.
\end{equation}
We now apply the stability of optimal transport guaranteed by the absolute continuity of $\nu$ which gives the lower bound $ \tilde{J}(z) - \tilde{J}(z_0) \geq \frac{1}{2\lambda}  \mathbb{E}_{\tilde{\nu}}[\|\nabla(z^* - z^*_0)\|^2]$.

\end{proof}

Note that this upper-bound encompasses the balanced case and that the entropy can bring extra convexity. In the generic case, we shall write $J(z) - J(z_0) \geq d_{H^\circ}^\lambda(z, z_0)$ where the pseudo distance $d_{H^\circ}^\lambda$ is defined as
\begin{equation*}
\begin{split}
      d_{H^\circ}^\lambda(f, g)^2 \coloneqq ~ & \frac{1}{2 \lambda}\mathbb{E}_{\tilde{\nu}}[\| \nabla (f^* - g^*)\|^2] +  C' (\mathbb{E}_{\nu}[(f^* - g^*)^2] + \mathbb{E}_{\mu}[(f-g)^2]) \, ,
\end{split}
\end{equation*}
where $C'\geq 0$. When $C'$ is strictly positive, $J$ is (locally, at the optimum) strongly convex with respect to an $H^1$ norm while in the balanced setting, there it is formulated with the semi-norm $\dot{H}^1$. Furthermore, when $C'>0$, $J(z) - J(z_0)$  not only controls $z_0^* - z^*$ in an $L^2$ sense but also $z_0 - z$ in an $L^2$ sense. This is a notable improvement with respect to balanced OT: in this case, in order to upper-bound $z-z_0$ with $J(z) - J(z_0)$, a smoothness assumption must be made on $z_0$; in the UOT case, if the entropy is sufficiently convex, we can obtain an upper-bound \textit{without} this extra smoothness assumption. In the next section, we use this property to apply a localization technique and obtain fast rates without requiring smoothness of the functions in $C$.

\section{Estimation of UOT maps}

In this section, we restrict ourselves to the case where $\mu, \nu$ are probability measures which we only access through their $n$-samples stochastic counterparts $\hat{\mu}, \hat{\nu}$. Even though this setting is much more restrictive than the original UOT setting where $\mu$ and $\nu$ can be arbitrary positive radon measures, it remains nonetheless a relevant setting in ML applications for instance. Indeed, the relaxation of the hard marginal OT constraint by a divergence allow to better handle outliers as shown experimentally by \citet{mukherjee2021outlier}.

In this stochastic setting, a natural way to estimate UOT map is to solve the empirical semi-dual over a given search space $C$. 

\begin{definition}[Stochastic Semi-Dual Unbalanced OT]
Let $C$ be a set of real-valued function, we define $\widehat{\text{UOT}}_C$
\begin{equation}
\label{EqEmpiricalUOT}
    \widehat{\text{UOT}}_C = - \inf_{z \in C} \hat{J}(z) \, ,
\end{equation}
where $\hat{J} = J_{\hat{\mu}, \hat{\nu}}$. Conversely,we define the empirical potential
\begin{equation}
    \hat{z}_C = \arg \min_{z \in C} \hat{J}(z) \, .
\end{equation}
When no confusion is possible, we shall simply denote it $\hat{z}$.
\end{definition}

If the true unbalanced potential $z_0$ belongs to $C$, we can prove that the empirical potential $\hat{z}$ converges toward $z_0$ with respect to $d_{H^\circ}^\lambda$ at a rate that will depend on the complexity of $C$. 

\subsection{Generic case}

We show that under suitable assumptions, the solutions of empirical unbalanced semi-dual OT converges toward the ground truth in the $d_{H^\circ}^\lambda$ sense.
\begin{assumption}
\label{assump:compact}
The measures $\mu$, $\nu$ have support included in $B_R$, where $B_r$ is the euclidean ball of $\mathbb{R}^d$ centered in $0$ and of radius $r$.
\end{assumption}
\begin{assumption}
\label{assump:boundedness}
The measures $\mu$, $\nu$ have densities with respect to the Lebesgue measure on $B_R$.
\end{assumption}
\begin{assumption}
\label{assump:z_0}
There exists $\tilde{z}_0 \in C$ such that $\tilde{z}_0$ coincides with $z_0$ on $\text{supp}(\mu)$ and with $\tilde{z}_0^*$ coincides with $z_0^*$ on $\text{supp}(\nu)$.
\end{assumption}
\begin{assumption}
\label{assump:C}
The functions in $C$ are uniformly bounded by $M(r)$ over $B_r$, uniformly lower bounded by $l$ and are $\lambda$-strongly convex.
\end{assumption}

The goal of Assumption \ref{assump:boundedness} is to ensure the existence of the unbalanced transport map between $\mu$, $\nu$. The goal of Assumption \ref{assump:z_0} is to ensure the absence of bias in the model. We believe that under a finer analysis, Assumption \ref{assump:compact} could be replaced with sub-gaussian measures. We show in the following Lemma that Assumption \ref{assump:C} ensures that the conjugate of the functions in $C$ are both bounded and smooth on every ball. 

\begin{lemma}
\label{lemma:boundC^*}
For all $z$ that are $\lambda$-strongly convex and such that $z\geq l$, $\|z\|_{L^\infty_{B_r}} \leq M(r)$, we have $\|\nabla z^*\|_{L^\infty_{B_r}} \leq G(r) \coloneqq  \frac{r}{\lambda} + \sqrt{\frac{2(M(0) - l)}{\lambda}}$ and $\| z^* \|_{L^\infty_{B_r}} \leq M'(r) \coloneqq r G(r) + M(G(r))$.
\end{lemma}

\begin{proof}
For $z \in C$, we have that $z^*$ is $\frac{1}{\lambda}$-smooth. In particular, for $x \in B_r$
\begin{align}
    \| \nabla z^* (x) \| & = \| \nabla z^*(x) - \nabla z^*(0) + \nabla z^*(0) \| \\
    & \leq \| \nabla z^*(x) - \nabla z^*(0)\| + \| \nabla z^*(0) \| \\
    & \leq \frac{r}{\lambda} +  \| \nabla z^*(0) \| \, .
\end{align}
Now recall that $\nabla z^*(0) = \arg \min_{x \in \mathbb{R}^d} z(x) $. Since $z$ is $\lambda$-strongly convex, we have the following inequality
\begin{equation}
    z(0) \geq z(x_*) + \frac{\lambda}{2} \|x_*\|^2 \, ,
\end{equation}
where $x_* = \arg \min_{x \in \mathbb{R}^d} z(x)$. Using that $z(0) \leq M(0)$ and $-z \leq -l$, we recover
\begin{equation}
    \|x_*\| \leq \sqrt{\frac{2(M(0) - l)}{\lambda}} \, .
\end{equation}
The bound on $\| z^* \|_{L^\infty_{B_r}}$ follows the definition of the Fenchel-Legendre transform
\begin{equation}
    z^*(x) = x^\top \nabla z^*(x) - z(\nabla z^*(x)) \, .
\end{equation}
\end{proof}

Finally, combined with the previous Lemma, Assumption \ref{assump:C} also ensures that the conjugate on $C$ has a Lipschitz behavior with respect to the sup norm. 

\begin{lemma}
\label{lemma:lipfenchel} 
Let $z_1, z_2$ be $\lambda$-strongly convex functions such that $z_1, z_2$ are lower-bounded by $l$ and bounded by $M(r)$ on $B_r$. We have $\| z_1^* - z_2^* \|_{L^\infty_{B_R}} \leq \| z_1 - z_2 \|_{L^\infty_{B_{G(R)}}}$, where $G(r) \coloneqq \frac{r}{\lambda} + \sqrt{\frac{2(M(0) - l)}{\lambda}}$ as in Lemma \ref{lemma:boundC^*}.
\end{lemma}
\begin{proof}
Let $x \in B_R$. By definition of the Fenchel transform, we have for all $y \in \mathbb{R}^d$
\begin{equation}
    z_1^*(x) \geq x^\top y - z_1(y) \, ,
\end{equation}
with equality when $y = \nabla z_1^*(x)$. Hence, we have for all $y$
\begin{equation}
    z_1^*(x) - z_2^*(x) \geq x^\top y - z_1(y) + z_2(\nabla z_2^*(x)) - x^\top\nabla z_2^*(x) \, .
\end{equation}
In particular, for $y = \nabla z_2^*(x)$, we obtain
\begin{equation}
    z_1^*(x) - z_2^*(x) \geq  z_2(\nabla z_2^*(x)) - z_1(\nabla z_2^*(x)) \, ,
\end{equation}
and applying Lemma \ref{lemma:boundC^*} yields $z_1^*(x) - z_2^*(x)\geq - \|z_1 - z_2 \|_{L^\infty_{B_{G(R)}}}  $. Conversely, flipping the role of $z_1, z_2$, we obtain
\begin{equation}
    z_2^*(x) - z_1^*(x) \geq  z_1(\nabla z_1^*(x)) - z_2(\nabla z_1^*(x)) \, ,
\end{equation} 
which yields $|z_1^*(x) - z_2^*(x)| \leq \|z_1 - z_2 \|_{L^\infty_{B_G(R)}}$.
\end{proof}

We have now all the ingredients to derive our result.

\begin{proposition}
\label{propErrormaps}
Denoting $\hat{z}_C$ the solution of problem \eqref{EqEmpiricalUOT}, we have under Assumptions 1-4, if the unbalanced optimal transport potential $z_0$ between $\mu$ and $\nu$ belongs to $C$, then we have for all $\delta \leq \frac{M'}{L}$
\begin{equation*}
    \mathbb{E}[d_{H^\circ}^\lambda(\hat{z}_C, z_0)^2] \lesssim \delta + \frac{1}{\sqrt{n}} \int_{\frac{\delta}{4}}^{\frac{M'}{L}} \sqrt{n(C, L^\infty_{B_{R'}}, Lu)} \mathrm{d} u \, 
\end{equation*}
where $n(C, \| \cdot \|, u)$ is the logarithm of the covering number, also called the \textit{metric entropy}, of $C$ with respect to the $\| \cdot \|$ (semi)-norm at scale $u$, $M'=(M, R, \lambda, l, \varphi)$, $R' = (M, R, \lambda, l)$, $L = (M, R, \lambda, l, \varphi)$ and $\lesssim$ hides a factor $64$. 
\end{proposition}

\begin{proof}

We start by applying the strong convexity inequality of the semi-dual and the optimality conditions 
\begin{align}
    d_{H^\circ}^\lambda(\hat{z}, z_0)^2 & \leq J(\hat{z}) - J(z_0) \\
      &  = J(\hat{z}) - \hat{J}(\hat{z}) + \hat{J}(\hat{z}) - \hat{J}(z_0)
      + \hat{J}(z_0) - J(z_0) \, .
\end{align}
Using Assumption \ref{assump:z_0}, the term $\hat{J}(\hat{z}) - \hat{J}(z_0)$ is negative hence we have
\begin{align}
    d_{H^\circ}^\lambda(\hat{z}, z_0)^2 & \leq J(\hat{z}) - \hat{J}(\hat{z}) + \hat{J}(z_0) - J(z_0) \\ 
    & \label{eq:Cterm} \leq \sup_{z \in C} \langle \phi^*(z - q), \mu - \hat{\mu} \rangle \\
    & \label{eq:C^*term} ~ + \sup_{z \in C^*} \langle \phi^*(z - q), \nu - \hat{\nu} \rangle  \\
    & \label{eq:z_0term} ~ + \hat{J}(z_0) - J(z_0) \, ,
\end{align}
where we denoted $C^* = \{z^*, z \in C\}$.
\paragraph{Bound on term \eqref{eq:Cterm}} Denoting $C_0 = \{\phi^*(g - q), g \in C \}$, we apply \citet[Theorem 16] {Luxburg04Distance} to bound our empirical process 
\begin{equation*}
    W \coloneqq \sup_{z \in C} \langle \phi^*(z - q), \mu - \hat{\mu} \rangle \, ,
\end{equation*}
and we obtain for all $\delta > 0$
\begin{equation}
    \mathbb{E}[W] \leq 2\delta + \frac{4\sqrt{2}}{\sqrt{n}} \int_{\frac{\delta}{4}}^\infty \sqrt{n(C_0, L^2(\hat{\mu}), u)} \ud u \, .
\end{equation}
Noting that $\|g\|_{L^2(\hat{\mu})} \leq \|g\|_{L^\infty(\mu)}$ almost surely, we recover the upper bound
\begin{equation}
    \mathbb{E}[W] \leq 2\delta + \frac{4\sqrt{2}}{\sqrt{n}} \int_{\frac{\delta}{4}}^\infty \sqrt{n(C_0, L^\infty(\mu), u)} \ud u \, .
\end{equation}
Since the functions in $C$ are uniformly bounded by $M(R)$ on $B_R$ and that $\mu$ is supported on $B_R$, we have $\forall (g_1, g_2) \in C^2$,
\begin{equation}
    \| \phi^*(g_1 - q) - \phi^*(g_2 - q)\|_{L^\infty(\mu)} \leq L^1_{\phi^*} \| g_1 - g_2 \|_{L^\infty(\mu)} 
    \, ,
\end{equation}
where $L^1_{\phi^*}$ is defined as
\begin{equation}
\label{eqDefL^1}
    L^1_{\phi^*} \coloneqq \sup_{x \in [- M_1, M_1]} |\partial \phi^*(x) | \, ,
\end{equation}
and $M_1 = 2M(R) + R^2$. In particular, we get the new upper-bound for all $\frac{\delta}{4} \leq \frac{2M(R)}{L^1_{\phi^*}}$
\begin{align*}
    \mathbb{E}[W] & \leq 2 \delta + \frac{4\sqrt{2}}{\sqrt{n}} \int_{\frac{\delta}{4}}^{\frac{2M(R)}{L^1_{\phi^*}}} \sqrt{n(C, L^\infty(\mu), L^1_{\phi^*}u)} \ud u \\
    & \leq 2 \delta + \frac{4\sqrt{2}}{\sqrt{n}} \int_{\frac{\delta}{4}}^{\frac{2M(R)}{L^1_{\phi^*}}} \sqrt{n(C, L^\infty_{B_R}, L^1_{\phi^*}u)} \ud u \, .
\end{align*}

\paragraph{Bound on term \eqref{eq:C^*term}} Lemma \ref{lemma:boundC^*} ensures that the functions in $C^*$ are uniformly bounded on every ball $B_r$ by some constant $M'(r)$. In particular, we can proceed as in the last paragraph and obtain
\begin{equation*}
    \mathbb{E}[W^*] \leq 2 \delta + \frac{4\sqrt{2}}{\sqrt{n}} \int_{\frac{\delta}{4}}^{\frac{2M'(R)}{L^2_{\phi^*}}} \sqrt{n(C^*, L^\infty_{B_R}, L^2_{\phi^*}u)} \ud u \, ,
\end{equation*}
where $W^* \coloneqq \sup_{z \in C^*} \langle z, \nu - \hat{\nu} \rangle$ and $L^2_{\phi^*}$ is defined as
\begin{equation}
\label{eqDefL^2}
    L^2_{\phi^*} := \sup_{x \in [-M_2, M_2]} |\partial \phi^*(x) | \, ,
\end{equation}
with $M_2 = 2M'(R) + R^2$. Using Lemma \ref{lemma:lipfenchel} that states 
\begin{equation}
    \| z_1^* - z_2^* \|_{L^\infty_{B_R}} \leq \| z_1 - z_2 \|_{L^\infty_{B_{G(R)}}} \, ,
\end{equation}
for some constant $G(R)$, we can control the covering number of $C^*$ with respect to the $L^\infty_{B_R}$ and we have the upper-bound for $\frac{\delta}{4} \leq \frac{2M'(R)}{L^2_{\phi^*}}$
\begin{equation*}
     \mathbb{E}[W^*] \leq 2 \delta + \frac{4\sqrt{2}}{\sqrt{n}} \int_{\frac{\delta}{4}}^{\frac{2M'(R)}{L^2_{\phi^*}}} \sqrt{n(C, L^\infty_{B_{G(R)}}, L^2_{\phi^*}u)} \ud u \, .
\end{equation*}

\paragraph{Final upper bound} Since the term \eqref{eq:z_0term} is zero in average, we obtain our final bound
\begin{equation*}
    d_{H^\circ}^\lambda(\hat{z}, z_0)^2 \leq 4 \delta + \frac{8\sqrt{2}}{\sqrt{n}} \int_{\frac{\delta}{4}}^{\frac{M'}{L}} \sqrt{n(C, L^\infty_{B_{R'}}, Lu)} \mathrm{d} u \, ,
\end{equation*}
where $M' = 2 \max(M(R), M'(R))$ and $L = \max(L^1_{\phi^*}, L^2_{\phi^*})$
\end{proof}

Leveraging the recent regularity results on UOT derived in \citet{gallouet2021regularity}, we can deduce from Proposition \ref{propErrormaps} an upper-bound for the statistical estimation of UOT potentials. 

\begin{corollary}
\label{corolLowSmooth}
Assume that $\mu$ and $\nu$ have compact and convex support with densities $(f,g)$ bounded away from zero and infinity and assume that $\varphi$is strictly convex with infinite slope at $0$. If $(f,g)$ are $k$-times continuously differentiable with $k\in \mathbb{N}^\star$ then, denoting $z_0$ an optimal unbalanced OT potential, there exists $C$ such that the empirical potential $\hat{z}_C$ verifies
\begin{equation}
    \begin{cases}
    \mathbb{E}[d_{H^\circ}^\lambda(\hat{z}_C, z_0)^2] \lesssim n^{-\frac{k+2}{d}} \text{ if $k + 2 < d/2$} \, , \\
    \mathbb{E}[d_{H^\circ}^\lambda(\hat{z}_C, z_0)^2] \lesssim \frac{\log(n)}{\sqrt{n}} \text{ ~if $k + 2  = d/2$} \, , \\
    \mathbb{E}[d_{H^\circ}^\lambda(\hat{z}_C, z_0)^2] \lesssim \frac{1}{\sqrt{n}} \text{~~~~ if $k + 2 > d/2$} \, .
    \end{cases}
\end{equation}
\end{corollary}

\begin{proof}
Using the Corollary 9 of \citet{gallouet2021regularity}, we can ensure that $z_0, z_0^*$ are $(k + 2)$-times continuously differentiable over the support of $\mu$ and $\nu$ respectively. Recalling that for all $x \in \text{supp}(\nu)$
\begin{equation}
    \nabla^2 z_0(x) = [\nabla^2 z_0^*(\nabla z_0(x))]^{-1} \, ,
\end{equation}
and using the fact that $\nabla z_0$ is a diffeomorphism between the support $\mu$ and $\nu$, we recover that $z_0$ is $\lambda$-strongly convex over $\text{supp}(\mu)$ where we defined 
\begin{equation}
    \frac{1}{\lambda} \coloneqq \underset{y \in \text{supp}(\nu)}{\sup} \|\nabla^2 z_0^*(y)\| \, .
\end{equation}
Now, recall that in order to apply our previous result, we need to globally bound the strong-convexity constant as well as controlling the sup norm over every ball. To achieve this, we can extend these potentials to the whole domain. Proposition 1.5 in \citet{azagra2019smooth} provides a $(k+2)$-times continuously differentiable convex extension $\tilde{g}_0$ of $z_0 - \lambda q$ on the whole domain $\mathbb{R}^d$. Defining $\tilde{z}_0 = \tilde{g}_0 + \lambda q$, we have that $\tilde{z}_0$ coincides with $z_0$ on $\text{supp}(\mu)$. Using again the diffeomorphism property of $\nabla z_0$ between $\text{supp}(\mu)$ and $\text{supp}(\nu)$, we have that $\tilde{z}_0^*$ coincides with $z_0^*$ on $\text{supp}(\nu)$. Now let us define 
\begin{align*}
    C = \{z \mid & \|z\|_{L_{B_r}^\infty} \leq \|\tilde{z}_0\|_{L_{B_r}^\infty}, \|\nabla^{k+2} z\|_{L^\infty_{B_r}} \leq  \|\nabla^{k+2} \tilde{z}_0\|_{L^\infty_{B_r}}, z\geq l,  \\
    & z \text{ is } \text{$\lambda$-strongly convex} \} \, ,
\end{align*}
where $l$ is the minimum of $\tilde{z}_0$. The set $C$ indeed meets Assumption \ref{assump:C} and Assumption \ref{assump:z_0} hence we can apply Prop. \ref{propErrormaps} which yields
\begin{equation}
    \mathbb{E}[d_{H^\circ}^\lambda(\hat{z}_C, z_0)^2] \lesssim \delta + \frac{1}{\sqrt{n}} \int_{\frac{\delta}{4}}^{\frac{M'}{L}} \sqrt{n(C, L^\infty_{B_{R'}}, Lu)} \mathrm{d} u \, .
\end{equation}
Finally, using \citet[Theorem 2.7]{vaart1996weak}, we have $n(C, L^\infty_{B_{R'}}, Lu) \lesssim u^{-\frac{d}{k+2}}$. If $\frac{k+2}{d} < 1/2$, take $\delta = n^{-\frac{k+2}{d}}$. For this choice of $\delta$,
\begin{align}
    \frac{1}{\sqrt{n}} \int_{\frac{\delta}{4}}^{\frac{M'}{L}} \sqrt{n(C, L^\infty_{B_{R'}}, Lu)} \mathrm{d} u & \lesssim \frac{1}{\sqrt{n}} (n^{-\frac{k+2}{d}})^{1 - \frac{d}{2(k+2)}} \\
    & \lesssim  \frac{1}{\sqrt{n}} n^{-\frac{2(k+2) -d}{2d}} \\
    & = n^{-\frac{k+2}{d}} \, .
\end{align}
If $\frac{k+2}{d} = 1/2$, take $\delta = \frac{1}{\sqrt{n}}$. For this choice of $\delta$, the integral is of order $\log(n)$ which yields the upper-bound
\begin{equation}
    \mathbb{E}[d_{H^\circ}^\lambda(\hat{z}_C, z_0)^2] \lesssim \frac{\log(n)}{\sqrt{n}} \, .
\end{equation}
Finally, if $\frac{k+2}{d} > 1/2$, taking $\delta=0$ yields
\begin{equation}
    \mathbb{E}[d_{H^\circ}^\lambda(\hat{z}_C, z_0)^2] \lesssim \frac{1}{\sqrt{n}} \, .
\end{equation}
\end{proof}

We see that we have two distinct regimes: a regime where the extra smoothness directly improves the rate of estimation and a highly smooth regime where the rate saturates at $1/\sqrt{n}$. In particular, we do not recover the asymptotic rate (with respect to the smoothness) $1/n$ which in known to be minimax in the case of balanced OT (see \citet{hutter2021Minimax}). We show in the next paragraph that this $1/\sqrt{n}$ rate can be improved under suitable assumptions on the metric entropy of $C$.

\subsection{Low metric entropy case}

When the metric entropy of $C$ slowly diverges, we can obtain faster rates than $1/\sqrt{n}$. The central argument is the localization: thanks to the strong convexity of the semi-dual, we can localize the empirical potential $\hat{z}$ in a certain neighborhood of the ground truth $z_0$. The next assumption allows to apply the stability result of the semi-dual with a $L^2$ control on $\hat{z} - z_0$ and $\hat{z}^* - z_0^*$; From the statistical point of view, it allows to employ localization arguments.
\begin{assumption}
The conjugate of the entropy $\varphi^*$ is strongly convex on every compact.
\end{assumption}

Hence, instead of controlling the global empirical processes
\begin{equation}
\label{eqGlobalProcess}
    W = \sup_{z \in C} ~ \langle z, \mu - \hat{\mu} \rangle \, ,
\end{equation}
and 
\begin{equation}
\label{eqGlobalProcess_conj}
    W^* = \sup_{z \in C} ~ \langle z^*, \nu - \hat{\nu} \rangle \, ,
\end{equation}
we simply must control the localized empirical processes
\begin{equation}
\label{eqLocalProcess}
    W(\tau) \coloneqq \sup_{z \in C\cap B^{\circ}(z_0, \tau)} \langle \phi^{*}(z - q) - \phi^{*}(z_0 - q), \mu - \hat{\mu} \rangle \, ,
\end{equation}
and 
\begin{equation}
\label{eqLocalProcess_conj}
    W^*(\tau) \coloneqq \sup_{z \in C\cap B^{\circ}(z_0, \tau)} \langle \phi^{*}(z^* - q) - \phi^{*}(z_0^* - q), \nu - \hat{\nu} \rangle \, ,
\end{equation}
where $\tau$ is a suitable radius and  $B^\circ(z_0, \tau)$ is the ball centered in $z_0$ of radius $\tau$ with respect to the $d_{H^\circ}^\lambda$ pseudo-distance. In the case where the metric entropy of $C$ grows too fast, the localized processes \eqref{eqLocalProcess} and \eqref{eqLocalProcess_conj} behave like the global processes \eqref{eqGlobalProcess} and \eqref{eqGlobalProcess_conj} and we cannot apply localization. However, when the metric entropy of $C$ is sufficiently low, the following lemma shows that $W(\tau)$ and $W^*(\tau)$ are upper bounded with high probability by $\frac{\tau^{1-\alpha/2}}{\sqrt{n}}$ with $2 > \alpha > 0$. 

\begin{lemma}
\label{lemmaLocalProcesses}
Under Assumptions 4-5, if we assume that there exists $(P_\mu, P_\nu)$ and $\alpha < 2$ such that for every $u \in \mathbb{R}_{\geq 0}$, $n(C, L^2(\mu), u) \leq P_{\mu} u^{-\alpha}$ and $n(C, L^2(\nu), u) \leq P_{\nu} u^{-\alpha}$, it holds with probability at least $1 - e^{-t}$ 
\begin{equation}
    \begin{cases}
W(\tau) \leq \frac{8 \sqrt{2P_{\mu}}}{(1 - \frac{\alpha}{2})\sqrt{n(L^1_{\phi^*})^{\alpha}}}(K\tau)^{1 - \alpha/2} + K\tau \sqrt{\frac{2t}{n}} + \frac{2 M(R)L^1_{\phi^*}}{n} \\
W^*(\tau) \leq \frac{8 \sqrt{2P_{\nu}}}{(1 - \frac{\alpha}{2})\sqrt{n(L^2_{\phi^*})^{\alpha}}}(K'\tau)^{1 - \alpha/2} + K'\tau \sqrt{\frac{2t}{n}} + \frac{2 M'(R)L^2_{\phi^*}}{n}  \, ,
\end{cases}
\end{equation}
where $L^1_{\phi^*}, L^2_{\phi^*}$ are defined in Equations \eqref{eqDefL^1} and \eqref{eqDefL^2} respectively and measure local lipschitz behaviors of $\varphi^*$, $M(R)$ is defined in Assumption \ref{assump:boundedness} and is a uniform bound over $B_R$ of the potentials in $C$, $M'(R)$ is defined in Lemma \ref{lemma:boundC^*} and is a uniform bound over $B_R$ of the conjugate of the potentials in $C$, and $K = K(R, M, \phi^*)$, $K'=K'(R, M, \phi^*, \lambda, l)$ are the embedding constants of $(C, L^2(\mu))$ and $(C^*, L^2(\nu))$ in $H^\circ$.
\end{lemma}

\begin{proof}

The proof relies on the Lipschitz behavior of the Legendre transform that preserves the metric entropy of $C$ and on the Bousquet concentration inequality. We start by analyzing the term $W(\tau)$.

\paragraph{Term $W(\tau)$}
Let us denote $C_0 = \{\phi^{*}(z - q) - \phi^{*}(z_0 - q), z \in C\cap B^{\circ}(z_0, \tau)\}$. For $g \in C_0$ of the form $g = \phi^{*}(z - q) - \phi^{*}(z_0 - q)$ with $z \in C\cap B^{\circ}(z_0, \tau)$, we have the pointwise bound for all $x \in B_R$,
\begin{equation}
\label{eqPointWiseUB}
    |g(x)| \leq L_{\phi^*}^1 |z(x) - z_0(x)| \, ,
\end{equation}
where $L^1_{\phi^*} := \sup_{x \in [- M_1, M_1]} |\partial \phi^*(x) |$ with $M_1 = 2M(R) + R^2$ as in the previous proof. This implies $\| g \|_{L^2(\mu)} \leq L_{\phi^*}^1 \|z - z_0\|_{L^2(\mu)}$. Since we assumed $\phi^*$ strongly convex on every compact, there exists $K = K(R, M, \phi^*) > 0$ such that $\|z - z_0\|_{L^2(\mu)} \leq K d_{H^\circ}^\lambda(z, z_0) $ and in particular, all $g \in C_0$ verifies $\|g\|_{L^2(\mu)} \leq K \tau$. Hence, applying \citet[Theorem 16]{Luxburg04Distance}, we obtain for all $\frac{\delta}{4}\leq K \tau$
\begin{equation}
    \mathbb{E}[W(\tau)] \leq 2 \delta + \frac{4\sqrt{2}}{\sqrt{n}} \int_{\frac{\delta}{4}}^{K\tau} \sqrt{n(C_0, L^2(\mu), u)} \ud u \, .
\end{equation}
Again, taking $(g_1, g_2) \in C_0^2$ of the form $g_1 =  \phi^{*}(z_1 - q) - \phi^{*}(z_0 - q)$ and $g_2 =  \phi^{*}(z_2 - q) - \phi^{*}(z_0 - q)$ with $(z_1, z_2) \in (C\cap B^{\circ}(z_0, \tau))^2$, we have
\begin{equation}
    \| g_1 - g_2 \|_{L^2(\mu)} \leq L^1_{\phi^*} \|z_1 - z_2 \|_{L^2(\mu)} \, ,
\end{equation}
and in particular, we recover the upper-bound
\begin{equation}
     \mathbb{E}[W(\tau)] \leq 2 \delta + \frac{4\sqrt{2}}{\sqrt{n}} \int_{\frac{\delta}{4}}^{K\tau} \sqrt{n(C, L^2(\mu),  L^1_{\phi^*}u)} \ud u .
\end{equation}
Now, we assumed that for all $u \in \mathbb{R}^{+}$ we had the upper-bound, $n(C, L^2(\mu), u) \leq P_{\mu} u^{-\alpha}$ with $\alpha < 2$, we obtain taking $\delta = 0$ our final upper bound
\begin{equation}
    \mathbb{E}[W(\tau)] \leq \frac{4 \sqrt{2P_{\mu}}}{(1 - \frac{\alpha}{2})\sqrt{n(L^1_{\phi^*})^{\alpha}}}(K\tau)^{1 - \alpha/2} \, .
\end{equation}
There remains to bound the process $W(\tau)$ with high probability. We use for this the Bousquet concentration inequality.
\begin{lemma}
[Bousquet, see Theorem 26 in \citet{hutter2021Minimax}]
Let $\mathcal{F}$ be a class of functions such that for every $f \in \mathcal{F}$, $\|f\|_{L^2(\mu)}^2 \leq \sigma^2$ and $\|f\|_{L^\infty(\mu)} \leq M$, then for all $t>0$, we have with probability at least $1 - e^{-t}$
\begin{equation}
    \sup_{f \in F} \sqrt{n}|\langle f, \mu - \hat{\mu} \rangle | \leq 2 \mathbb{E}[\sup_{f \in F} \sqrt{n}|\langle f, \mu - \hat{\mu} \rangle |] + \sigma \sqrt{2t} + \frac{M}{\sqrt{n}}t \, .
\end{equation}
\end{lemma}

Applying this result to $W(\tau)$ yields that with probability at least $1-e^{-t}$,
\begin{equation}
    W(\tau) \leq \frac{8 \sqrt{2P_{\mu}}}{(1 - \frac{\alpha}{2})\sqrt{n(L^1_{\phi^*})^{\alpha}}}(K\tau)^{1 - \alpha/2} + K\tau \sqrt{\frac{2t}{n}} + \frac{2t M(R)L^1_{\phi^*}}{n} \, ,
\end{equation}
where we used the pointwise upper-bound \eqref{eqPointWiseUB} and where $M(R)$ is the constant such that $\forall z \in C, \|z\|_{L^\infty_{B_R}} \leq M(R)$. 

\paragraph{Term $W^*(\tau)$} We can apply the same reasoning as previously. Indeed, as shown in Lemma \ref{lemma:boundC^*}, there exists a constant $M'(R)$ such that for all $z \in C$, $\|z^*\|_{L^\infty_{B_R}} \leq M'(R)$. In particular, since the potentials $z^*$ are bounded, we can also leverage the local strong convexity of $\phi^*$ that yields a constant $K' = K'(R, M, \phi^*, \lambda, l) > 0$ such that for every $z \in C, \|(z-z_0)^*\|_{L^2(\nu)} \leq K' d_{H^\circ}^\lambda(z, z_0)$. Hence we recover that with probability at least $1 - e^{-t}$,
\begin{equation}
    W^*(\tau) \leq \frac{8 \sqrt{2P_{\nu}}}{(1 - \frac{\alpha}{2})\sqrt{n(L^2_{\phi^*})^{\alpha}}}(K'\tau)^{1 - \alpha/2} + K'\tau \sqrt{\frac{2t}{n}} + \frac{2t M'(R)L^2_{\phi^*}}{n} \, .
\end{equation}

\end{proof}

Then, the informal reasoning is as follows: if we take $\tau = d_{H^\circ}^\lambda(\hat{z}, z_0)$, we have the upper bound: $\tau^2 \leq W(\tau) + W^*(\tau) \leq \frac{\tau^{1 - \alpha/2}}{\sqrt{n}}$ which constrains $\tau$ as $\tau^2 \leq n^{-\frac{1}{1+\alpha/2}}$.
\begin{proposition}
\label{PropFastRates}
Under Assumptions 1-5, if we assume that there exists $(P_\mu, P_\nu)$ and $\alpha < 2$ such that for every $u \in \mathbb{R}_{\geq 0}$, $n(C, L^2(\mu), u) \leq P_{\mu} u^{-\alpha}$ and $n(C, L^2(\nu), u) \leq P_{\nu} u^{-\alpha}$ then $\forall n \geq 1$, 
\begin{equation}
    \mathbb{E}[d_{H^\circ}^\lambda(\hat{z}, z_0)^2] \lesssim n^{-\frac{1}{1 + \alpha/2}} \, ,
\end{equation}
where $\lesssim$ hides constants that do not depend on $n$.
\end{proposition}

\begin{proof}

For $\tau > 0$, define $s=\frac{\tau}{\tau +  d_{H^\circ}^\lambda(\hat{z}, z_0)}$ and $\hat{z}_s = (1-s)z_0 + s \hat{z}$. By local strong convexity of $J$, we have 
\begin{equation}
     d_{H^\circ}^\lambda(\hat{z}_s, z_0)^2 \leq J(\hat{z}_s) - J(z_0) \, .
\end{equation}
Let us decompose the right hand side as $J(\hat{z}_s) - \hat{J}(\hat{z}_s) - (J(z_0) - \hat{J}(z_0)) + \hat{J}(\hat{z}_s) - \hat{J}(z_0)$. By convexity of $\hat{J}$, the last term can be upper-bounded by $s\hat{J}(\hat{z}) + (1-s)\hat{J}(z_0) - \hat{J}(z_0) = s(\hat{J}(\hat{z}) - \hat{J}(z_0))$. Since $\hat{z}$ is the minimizer of the empirical semi-dual, we have in particular that $s(\hat{J}(\hat{z}) - \hat{J}(z_0)) \leq 0$ which gives
\begin{align*}
     d_{H^\circ}^\lambda(\hat{z}_s, z_0)^2 & \leq  J(\hat{z}_s) - \hat{J}(\hat{z}_s) - (J(z_0) - \hat{J}(z_0))
     \\
     & = \langle \phi^*(\hat{z}_s - q) -  \phi^*(z_0 - q), \mu - \hat{\mu}\rangle + \langle \phi^*(\hat{z}_s^* - q) -  \phi^*(z_0^* - q), \nu - \hat{\nu}\rangle \, .
\end{align*}
Now, since $d_{H^\circ}^\lambda(\hat{z}_s, z_0) = \frac{\tau d_{H^\circ}^\lambda(\hat{z}, z_0)}{\tau + d_{H^\circ}^\lambda(\hat{z}, z_0)} \leq \tau$, we recover in the end $d_{H^\circ}^\lambda(\hat{z}_s, z_0)^2 \leq W(\tau) + W^*(\tau) $.

Let us now consider $A=\{\tau, d_{H^\circ}^\lambda(\hat{z}, z_0) \geq \tau \}$. We wish to recover an upper-bound on $A$. Remark that $A = \{\tau, d_{H^\circ}^\lambda(\hat{z}_s, z_0) \geq \frac{\tau}{2} \}$. In particular, every $\tau \in A$ verifies with probability at least $1 - e^{-t}$
\begin{equation}
    \frac{\tau^2}{4} \leq \kappa \frac{\tau^{1 - \alpha/2}}{\sqrt{n}} + (K + K')\tau\sqrt{\frac{2t}{n}} + \frac{t\kappa'}{n} \, , 
\end{equation}
where $\kappa$ and $\kappa'$ are given in Lemma \ref{lemmaLocalProcesses} defined as 
\begin{equation}
    \begin{cases}
    \kappa = \frac{8\sqrt{2}}{(1-\frac{\alpha}{2})}\biggl[\frac{\sqrt{P_\mu}K^{1-\alpha/2}}{(L^1_{\varphi^*})^\frac{\alpha}{2}} + \frac{\sqrt{P_\nu}(K')^{1-\alpha/2}}{(L^2_{\varphi^*})^\frac{\alpha}{2}}\biggr] \\
    \kappa' = 2(M(R) L^1_{\varphi^*} + M'(R) L^2_{\varphi^*}) \, .
    \end{cases}
\end{equation}
. Let $A_n = \{\tau \in A, \tau \geq \frac{1}{\sqrt{n}} \}$. For $\tau \in A_n$, we have 
\begin{equation}
    \frac{\tau^2}{4} \leq \kappa \frac{\tau^{1 - \alpha/2}}{\sqrt{n}} + (K + K')\tau\sqrt{\frac{2t}{n}} + \frac{t \kappa'\tau}{\sqrt{n}} \, .
\end{equation}
Assuming that $t\geq 1$, we have two cases
\paragraph{Case 1} If $\tau \leq 1$, we have
\begin{equation}
    \frac{\tau^2}{4} \leq \frac{t \eta \tau^{1 - \alpha/2}}{\sqrt{n}} \, ,
\end{equation}
where $\eta = (\kappa + \kappa' + \sqrt{2}(K + K') )$ and we recover $\tau \leq \frac{(4\eta t)^{\frac{1}{1 + \alpha/2}}}{n^{\frac{1}{2+\alpha}}}$.
\paragraph{Case 2} If $\tau \geq 1$, we have
$\frac{\tau^2}{4} \leq \frac{t \eta \tau}{\sqrt{n}}$ \textit{i.e.} $\tau \leq \frac{4t \eta}{\sqrt{n}}$. 
\\\\
In any case, for $t\geq 1$, we have with probability at least $1-e^{-t}$
\begin{equation}
    \sup(A) \leq \frac{(4\eta' t)^{\frac{1}{1 + \alpha/2}} + (4\eta' t) }{n^{\frac{1}{2+\alpha}}} \, ,
\end{equation}
where we defined $\eta' = \max(\eta, 1)$. Now, by definition of $A$, we have for all $\epsilon > 0$, $d_{H^\circ}^\lambda(\hat{z}, z_0) \leq \sup(A) + \epsilon$. Taking $\epsilon \to 0$ gives that with probability at least $1 - e^{-t}$, for $t\geq 1$
\begin{align}
    d_{H^\circ}^\lambda(\hat{z}, z_0) & \leq \frac{(4\eta' t)^{\frac{1}{1 + \alpha/2}} + (4\eta' t) }{n^{\frac{1}{2+\alpha}}} \\
    & \leq \frac{8\eta' t}{n^{\frac{1}{2+\alpha}}} \, .
\end{align}
And in particular, $d_{H^\circ}^\lambda(\hat{z}, z_0)^2 \leq \frac{64(\eta')^2t^2}{n^{\frac{1}{1 + \alpha/2}}}$ with probability at least $1-e^{-t}$ for $t\geq 1$. We denote $X$ the random variable $d_{H^\circ}^\lambda(\hat{z}, z_0)^2$. Since $X$ is nonnegative almost surely, we can apply Fubini's formula
\begin{equation}
    \mathbb{E}[X] = \int_{0}^\infty P(X>u) \ud u \, .
\end{equation}
Let us make the change of variable $u = \frac{64(\eta')^2t^2}{n^{\frac{1}{1 + \alpha/2}}}$,
\begin{equation*}
    \mathbb{E}[X] = \frac{128(\eta')^2}{n^{\frac{1}{1 + \alpha/2}}} \biggl( \int_0^1 t P(X>  \frac{64(\eta')^2t^2}{n^{\frac{1}{1 + \alpha/2}}}) \ud t + \int_1^\infty t P(X>  \frac{64(\eta')^2t^2}{n^{\frac{1}{1 + \alpha/2}}}) \ud t \biggr)\, .
\end{equation*}
The integrand in the first term is upper-bounded by $1$ and the integrand on the second term is upper bounded by $te^{-t}$. Hence we obtain
\begin{align*}
    \mathbb{E}[d_{H^\circ}^\lambda(\hat{z}, z_0)^2] & \leq \frac{128(\eta')^2}{n^{\frac{1}{1 + \alpha/2}}} (1 + \int_{1}^\infty t e^{-t} dt ) \\
    & = \frac{128(1 + 2e^{-1})(\eta')^2}{n^{\frac{1}{1 + \alpha/2}}} \, .
\end{align*}
\end{proof}

An immediate consequence of this result is that we can improve the rate derived in Corollary \ref{corolLowSmooth} and recover the asymptotic behavior in $1/n$ when the smoothness grows.
\begin{corollary}
\label{corolHighSmooth}
Assume that $\mu$ and $\nu$ have compact and convex support with densities $(f,g)$ bounded away from zero and infinity and assume that $\varphi^*$ is strongly convex on every compact. If $(f,g)$ are $k$-times continuously differentiable with $k + 2 > d/2$ then, denoting $z_0$ an optimal unbalanced OT potential, there exists $C$ such that the empirical potential $\hat{z}_C$ verifies
\begin{equation}
    \mathbb{E}[d_{H^\circ}^\lambda(\hat{z}_C, z_0)^2] \lesssim n^{-\frac{1}{1 + \frac{d}{2(k +2)}}} \, .
\end{equation}
\end{corollary}
\begin{proof}
As in the proof of Corollary \ref{corolLowSmooth}, the original potential $z_0$ is $\lambda$-strongly convex and $(k+2)$-times continuously differentiable over $\text{supp}(\mu)$. It can be extended to $\tilde{z}_0$ that is also $\lambda$-strongly convex and $(k+2)$-times continuously differentiable but over the whole domain $\mathbb{R}^d$. The function $\tilde{z}_0$ is such that it coincides with $z_0$ on $\text{supp}(\mu)$ and its conjugate $\tilde{z}_0^\star$ coincides with $z_0^\star$ on $\text{supp}(\nu)$. We define the set 
\begin{align*}
    C = \{z \mid & \|z\|_{L_{B_r}^\infty} \leq \|\tilde{z}_0\|_{L_{B_r}^\infty}, \|\nabla^{k+2} z\|_{L^\infty_{B_r}} \leq  \|\nabla^{k+2} \tilde{z}_0\|_{L^\infty_{B_r}}, z\geq l,  \\
    & z \text{ is } \text{$\lambda$-strongly convex} \} \, ,
\end{align*}
where $l$ is the minimum of $\tilde{z}_0$. Theorem 2.7 of \citet{vaart1996weak} ensures that the metric entropy at scale $u$ of $C$ with respect to $L^2(\mu)$ and $L^2(\nu)$ is upper bounded by $u^{-\frac{k+2}{d}}$. Hence we can apply Prop. \ref{PropFastRates} and obtain 
\begin{equation}
    \mathbb{E}[d_{H^\circ}^\lambda(\hat{z}_C, z_0)^2] \lesssim n^{-\frac{1}{1 + \frac{d}{2(k +2)}}} \, .
\end{equation}
\end{proof}

\begin{figure}
    \centering
    \includegraphics[scale=0.4]{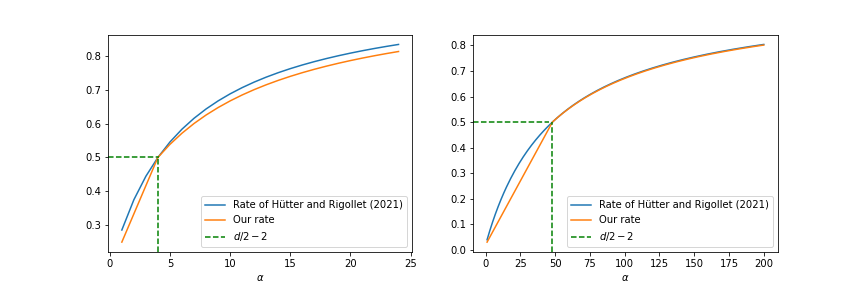}
    \caption{Comparison of our rates against the rates of Hutter and Rigollet (2021): on the left for $d=12$ and on the right for $d=100$.}
    \label{fig:my_label}
\end{figure}

Hence, using Corollary \ref{corolLowSmooth} and Corollary \ref{corolHighSmooth}, we obtain a rate of $n^{-\frac{\alpha+2}{d}}$ when $\alpha + 2 < d/2$ and $n^{-\frac{1}{1 + \frac{d}{2(\alpha+2)}}}$ when $\alpha + 2 > d/2$; note that we continuously transition from one rate to another when $\alpha + 2 = d/2$ where we recover (up to log factor) the parametric rate $1/\sqrt{n}$. Even though the fast rates do not encompass the balanced case as they require $\varphi^*$ to be at least locally strongly convex, we still compare our rates to the ones of \citet{hutter2021Minimax}. Under the same assumptions, they propose estimators that achieve a rate in $n^{-\frac{\alpha + 1}{\alpha + d/2}}$. As shown in Fig. 1 for the case $d=100$, their rate is faster for any $\alpha > 0$ yet when we transition in the highly smooth regime $\alpha + 2 > d/2$, our rate closely matches theirs. This discrepancy is due to the fact that we have no bias in our model \textit{i.e.} we assumed $z_0 \in C$. On their side, \citet{hutter2021Minimax} fixed $C$ to be a finite wavelet basis that does not necessarily contain $z_0$. In particular, they improve the bias variance trade-off on two levels: in the low smooth regime, they can benefit the acceleration given by the localization as they choose a small class of functions and for the same reason, in the highly smooth regime, they better leverage the localization.

\bibliography{references.bib}

\end{document}